\documentclass[final]{siamltex}

\usepackage[left=3cm, right=3cm, top=3cm, bottom=3cm]{geometry}

\usepackage{hyperref}
\usepackage{url}
\usepackage{amsmath,amssymb,bm,bbm}
\usepackage{amsfonts,dsfont}
\usepackage{graphics}
\usepackage{epsfig}
\usepackage{tablefootnote}
\usepackage{cite}

\newtheorem{thm}{Theorem}[section]  

\newtheorem{lem}[thm]{Lemma}

\newtheorem{corr}[thm]{Corollary}

\newtheorem{asmp}[thm]{Assumption}

\newcommand{\norm}[1]{\left|\left|#1\right|\right|}
\newcommand{\nn}{\nonumber}
\newcommand{\Real}{{\mathbb{R}}}
\newcommand{\Integer}{{\mathbb{Z}}}

\newcommand{\order}{\mathcal{O}}
\newcommand{\torder}{\widetilde{\mathcal{O}}}

\DeclareMathOperator*{\argmin}{arg\,min}

\newcommand{\prox}{\text{prox}}

\newcommand{\sj}{\sum_{j=(k-K)_+}^{k-1}}

\newcommand{\ndk}{\norm{d_k}}
\newcommand{\ndj}{\norm{d_j}}

\title{A Stronger Convergence Result on the Proximal Incremental Aggregated Gradient Method}

\author{
N.~D.~Vanli\thanks{Laboratory for Information and Decision Systems, Massachusetts Institute of Technology, Cambridge, MA 02139, USA. email: \{denizcan, mertg, asuman\}@mit.edu.}
\and M.~G\"urb\"uzbalaban\footnotemark[2]
\and A.~Ozdaglar\footnotemark[2]
}
\date{\today}

\begin{document}

\maketitle

\begin{abstract}
We study the convergence rate of the proximal incremental aggregated gradient (PIAG) method for minimizing the sum of a large number of smooth component functions (where the sum is strongly convex) and a non-smooth convex function. At each iteration, the PIAG method moves along an aggregated gradient formed by incrementally updating gradients of component functions at least once in the last $K$ iterations and takes a proximal step with respect to the non-smooth function. We show that the PIAG algorithm attains an iteration complexity that grows linear in the condition number of the problem and the delay parameter $K$. This improves upon the previously best known global linear convergence rate of the PIAG algorithm in the literature which has a quadratic dependence on $K$.
\end{abstract}

\section{Introduction}
We consider composite additive cost optimization problems, where the objective function is given by the sum of $m$ component functions $f_i(x)$ and a possibly non-smooth regularization function $r(x)$:
\begin{equation}\label{eq:goal}
  \min_{x\in\Real^n} F(x) \triangleq f(x) + r(x),
\end{equation}
where $f(x) = \frac{1}{m} \sum_{i=1}^m f_i(x)$. Each component function $f_i:\Real^n\to(-\infty,\infty)$ is assumed to be convex and continuously differentiable while the regularization function $r:\Real^n\to(-\infty,\infty]$ is proper, closed, and convex but not necessarily differentiable.

The recent paper \cite{dc} studied the PIAG algorithm, which at each iteration $k\geq0$, first constructs an \emph{aggregated gradient} defined by
\begin{equation}
  g_k \triangleq \frac{1}{m} \sum_{i=1}^m \nabla f_i(x_{\tau_{i,k}}), \nn
\end{equation}
where $\nabla f_i(x_{\tau_{i,k}})$ represents the gradient of the $i$th component function sampled at time $\tau_{i,k}$. This aggregated gradient is used to update $x_k$ as
\begin{equation}
  x_{k+1} = \prox_r^\eta(x_k - \eta g_k), \label{eq:update_rule}
\end{equation}
where the proximal mapping is defined as $\prox_r^\eta(y) = \argmin_{x\in\Real^n} \left\{ \frac{1}{2} \norm{x-y}^2 + \eta r(x) \right\}$ with a constant step size $\eta>0$.

It was shown in \cite{dc} that the PIAG algorithm attains a global linear convergence rate of $1-\torder(Q^{-1}K^{-2})$ in function suboptimality, i.e., $F(x_k)-F(x^*)$, where $x^*$ denotes the optimal solution to \eqref{eq:goal} and the tilde is used to hide the logarithmic terms in $Q$ and $K$. This result implies that in order to achieve an $\epsilon$-optimal solution, PIAG requires at most $\torder(QK^2\log(1/\epsilon))$ iterations. The independent work \cite{arda} also studied the PIAG algorithm and showed using a different analysis that it attains a global linear convergence rate of $1-\order(Q^{-1}K^{-2})$ in distance to the optimal solution $\norm{x_k-x^*}$. This result implies that to achieve a point in the $\epsilon$-neighborhood of the optimal solution, PIAG requires $\order(QK^2\log(1/\epsilon))$ iterations. The latter result on distances does not translate directly into a linear convergence rate in function suboptimality since the problem \eqref{eq:goal} is not smooth.

In this paper, by using the results presented in \cite{arda} and \cite{dc}, we provide a stronger linear convergence rate for the deterministic PIAG algorithm. In particular, in \cite{dc}, two lemmas regarding the relations on the proximal operator are introduced to provide a contraction relation on the function suboptimality. In \cite{arda}, a lemma that characterizes the linear convergence of a Lyapunov function is introduced, where the Lyapunov function satisfies a certain contraction relation with sufficiently small perturbation that depends on the recent history. By using these results, we prove that the PIAG algorithm attains a global linear convergence rate of $1-\order(Q^{-1}K^{-1})$ in function suboptimality. This implies that in order to achieve an $\epsilon$-optimal solution in suboptimality of the function values, PIAG requires at most $\order(QK\log(1/\epsilon))$ iterations. To our knowledge, this convergence rate result provides the best dependence on the condition number of the problem $Q$ and the delay parameter $K$ for deterministic incremental aggregated gradient methods.

\section{Assumptions}
Throughout the paper, we make the following standard assumptions that are used in both\cite{arda} and \cite{dc}.

\begin{asmp}\label{asmp:lips}\textbf{(Lipschitz gradients)}
Each $f_i$ has Lipschitz continuous gradients on $\Real^n$ with some constant $L_i \geq 0$, i.e.,
\begin{equation}
  \norm{\nabla f_i(x)-\nabla f_i(y)} \leq L_i \norm{x-y}, \nn
\end{equation}
for any $x,y\in\Real^n$.\footnote{If a function $f$ has Lipschitz continuous gradients with some constant $L$, then $f$ is called $L$-smooth. We use these terms interchangeably.}
\end{asmp}

Defining $L \triangleq \frac{1}{m} \sum_{i=1}^m L_i$, we observe that Assumption \ref{asmp:lips} and the triangle inequality yield
\begin{equation}
  \norm{\nabla f(x)-\nabla f(y)} \leq L \norm{x-y}, \nn
\end{equation}
for any $x,y\in\Real^n$, i.e., the function $f$ is $L$-smooth.

\begin{asmp}\label{asmp:conv}\textbf{(Strong Convexity)}
The sum function $f$ is $\mu$-strongly convex on $\Real^n$ for some $\mu>0$, i.e., the function $x \mapsto f(x) - \frac{\mu}{2} \norm{x}^2$ is convex.
\end{asmp}

\begin{asmp}\label{asmp:subdif}\textbf{(Subdifferentiability)}
The regularization function $r:\Real^n\to(-\infty,\infty]$ is proper, closed, convex and subdifferentiable everywhere in its effective domain, i.e., $\partial r(x) \neq \emptyset$ for all $x\in\{y\in\Real^n \, : \, r(y) < \infty \}$.
\end{asmp}

A consequence of Assumptions \ref{asmp:conv} and \ref{asmp:subdif} is that $F$ is strongly convex, hence there exists a unique optimal solution of problem \eqref{eq:goal}, which we denote by $x^*$ (cf. Lemma 6 in \cite{primaldual}).

Another consequence of Assumption \ref{asmp:subdif} is that the set of subgradients of $x_k$ is well-defined for all $k\geq0$. Then, it follows from the optimality conditions \cite{Bertsekas15Book} of the minimization problem in the proximal map in \eqref{eq:update_rule} that $0 \in \partial \phi(x_{k+1})$. This yields $x_{k+1}-(x_k - \eta g_k)+\eta h_{k+1}=0$, for some subgradient $h_{k+1} \in \partial r(x_{k+1})$. Thus, we can represent the update rule of the PIAG algorithm as follows
\begin{equation}
  x_{k+1} = x_k + \eta d_k, \nn
\end{equation}
where $d_k \triangleq -g_k-h_{k+1}$ is the direction of the update at time $k$.

\begin{asmp}\label{asmp:bdd_delay}\textbf{(Bounded Delay)}
Each component function is sampled at least once in the past $K\geq0$ iterations, i.e., there exists a finite integer $K$ such that $k-K \leq \tau_{i,k} \leq k$, for all $k\geq1$ and $i\in\{1,\dots,m\}$.
\end{asmp}

\section{Main Result}
In this section, we characterize the global linear convergence rate of the PIAG algorithm. Let
\begin{equation}\label{eq:lyapunov}
  F_k \triangleq F(x_k) - F(x^*)
\end{equation}
denote the suboptimality in the objective value at iteration $k$. The paper \cite{dc} presented two lemmas regarding the evolution of $F_k$ and $\norm{d_k}^2$. In particular, the first lemma investigates how the suboptimality in the objective value evolves over the iterations and the second lemma relates the direction of update to the suboptimality in the objective value at a given iteration $k$.

\begin{lem}\cite[Lemma 3.3]{dc}\label{prop1}
Suppose that Assumptions \ref{asmp:lips}-\ref{asmp:bdd_delay} hold. Then, the PIAG algorithm yields the following guarantee
\begin{equation}
  F_{k+1} \leq F_k - \frac{1}{2} \eta \norm{d_k}^2 + \eta^2 \frac{L}{2} \sj\ndj^2, \label{eq:prop1}
\end{equation}
for any step size $0<\eta\leq\frac{1}{L(K+1)}$.
\end{lem}

\begin{lem}\cite[Lemma 3.5]{dc}\label{prop2}
Suppose that Assumptions \ref{asmp:lips}-\ref{asmp:bdd_delay} hold. Then, for any $0<\eta\leq\frac{1}{L(K+1)}$, the PIAG algorithm yields the following guarantee
\begin{equation}
  -\ndk^2 \leq - \frac{\mu}{4} F_{k+1} + \eta L \sj \ndj^2. \nn
\end{equation}
\end{lem}

Before presenting the main result of this work, we introduce the following lemma, which was presented in \cite{arda}, in a slightly different form. This lemma shows linear convergence rate for a nonnegative sequence $Z_k$ that satisfies a contraction relation perturbed by shocks (represented by $Y_k$ in the lemma).

\begin{lem}\cite[Lemma 1]{arda}\label{mlsp}
Let $\{Z_k\}$ and $\{Y_k\}$ be a sequence of non-negative real numbers satisfying
\begin{equation}
  \alpha Z_{k+1} \leq Z_k - \beta \, Y_k + \gamma \sum_{j=k-A}^k Y_j, \label{eq:rec}
\end{equation}
for any $k\geq0$ for some constants $\alpha>1$, $\beta\geq0$, $\gamma\geq0$ and $A\in\Integer^+$. If
\begin{equation}
  \gamma (\alpha^{A+1}-1) \leq \beta (\alpha-1) \label{eq:asmp}
\end{equation}
holds, then $Z_k \leq \alpha^{-k} Z_0$, for all $k\geq0$.
\end{lem}


We next present the main theorem of this paper, which characterizes the linear convergence rate of the PIAG algorithm. In particular, we show that when the step size is sufficiently small, the PIAG algorithm is linearly convergent with a contraction rate that depends on the step size $\eta$ and the strong convexity constant $\mu$.

\begin{thm}\label{thm}
Suppose that Assumptions \ref{asmp:lips}-\ref{asmp:bdd_delay} hold. Then, the PIAG algorithm with step size $0<\eta\leq\frac{16}{\mu} \left[ \left( 1+\frac{1}{48Q} \right)^{\frac{1}{K+1}} - 1 \right]$ is linearly convergent satisfying
\begin{equation}
  F_k \leq \left( 1 + \eta \frac{\mu}{16} \right)^{-k} F_0, \label{eq:thm}
\end{equation}
for any $k\geq0$. Furthermore, if $\eta = \frac{16}{\mu} \left[ \left( 1+\frac{1}{48Q} \right)^{\frac{1}{K+1}} - 1 \right]$, then
	\begin{equation} F_k  \leq \left( 1-\frac{1}{49Q(K+1)} \right)^k F_0.
	\label{eq:thm-rate-cond-number}
	\end{equation}
\end{thm}

\begin{proof}
By Lemma \ref{prop2}, we have
\begin{equation}
  -\frac{1}{4} \eta \ndk^2 \leq - \eta \frac{\mu}{16} F_{k+1} + \eta^2 \frac{L}{4} \sj \ndj^2. \nn
\end{equation}
Using this inequality in \eqref{eq:prop1} of Lemma \ref{prop1}, we get
\begin{equation}
  \left( 1 + \eta \frac{\mu}{16} \right) F_{k+1} \leq F_k - \frac{1}{4} \eta \norm{d_k}^2 + \eta^2 \frac{3L}{4} \sj\ndj^2. \label{eq:apply}
\end{equation}
We want to apply Lemma \ref{mlsp} to the inequality \eqref{eq:apply} with $Z_k=F_k$ and $Y_k=\ndk^2$ for proving \eqref{eq:thm}. For this purpose, we need $0<\eta\leq\frac{1}{L(K+1)}$ in order for Lemma \ref{prop1} and Lemma \ref{prop2} to hold, and
\begin{equation}
  \eta^2 \frac{3L}{4} \left( \left( 1 + \eta \frac{\mu}{16} \right)^{K+1}-1 \right) \leq \frac{1}{4} \eta \left( \left( 1 + \eta \frac{\mu}{16} \right)-1 \right) \label{eq:mlsp_req}
\end{equation}
with $\eta>0$ for Lemma \ref{mlsp} to hold. Simplifying and rearranging terms in \eqref{eq:mlsp_req}, we obtain
\begin{equation}
  \left( 1 + \eta \frac{\mu}{16} \right)^{K+1}-1 \leq \frac{1}{48Q}. \nn
\end{equation}
Therefore, for any step size satisfying
\begin{equation}
  0 < \eta \leq \frac{16}{\mu} \left[ \left( 1+\frac{1}{48Q} \right)^{\frac{1}{K+1}} - 1 \right], \label{eq:mlsp_req2}
\end{equation}
Lemma \ref{mlsp} holds. We can also observe that the right-hand side of \eqref{eq:mlsp_req2} can be upper bounded using the Bernoulli inequality, i.e., $(1+x)^r\leq1+rx$ for any $x\geq-1$ and $r\in[0,1]$, as follows
\begin{align}
  \eta & \leq \frac{16}{\mu} \left( 1+\frac{1}{48Q(K+1)}-1 \right) \nn\\
    & = \frac{1}{3L(K+1)}.
\end{align}
Thus, the constraint \eqref{eq:mlsp_req2} satisfies the constraint $0<\eta\leq\frac{1}{L(K+1)}$ in Lemma \ref{prop1} and Lemma \ref{prop2} as well. Then, applying Lemma \ref{mlsp} to \eqref{eq:apply} yields \eqref{eq:thm}.

It remains to show \eqref{eq:thm-rate-cond-number}. Plugging in $\eta=\frac{16}{\mu} \left[ \left( 1+\frac{1}{48Q} \right)^{\frac{1}{K+1}} - 1 \right]$ in \eqref{eq:thm} of Theorem \ref{thm}, we have
\begin{align}
  F_k & \leq \left( 1 + \frac{16}{\mu} \left[ \left( 1+\frac{1}{48Q} \right)^{\frac{1}{K+1}} - 1 \right] \frac{\mu}{16} \right)^{-k} F_0 \nn\\
    & = \left( 1+\frac{1}{48Q} \right)^{\frac{-k}{K+1}} F_0 \nn\\
    & \leq \left( 1-\frac{1}{49Q} \right)^{\frac{k}{K+1}} F_0 \nn\\
    & \leq \left( 1-\frac{1}{49Q(K+1)} \right)^k F_0,
\end{align}
where the third line follows as $Q\geq1$. This implies \eqref{eq:thm-rate-cond-number} and completes the proof.
\end{proof}

We next introduce the following corollary, which highlights the main result of the paper. This corollary indicates that for an appropriately chosen step size, the PIAG algorithm is guaranteed to return an $\epsilon$-optimal solution after $\order(QK\log(1/\epsilon))$ iterations.

\begin{corr}\label{corr}
Suppose that Assumptions \ref{asmp:lips}-\ref{asmp:bdd_delay} hold. Then, the PIAG algorithm with step size $\eta = \frac{16}{\mu} \left[ \left( 1+\frac{1}{48Q} \right)^{\frac{1}{K+1}} - 1 \right]$ is guaranteed to return an $\epsilon$-optimal solution after at most $49Q(K+1)\log(F_0/\epsilon)$ iterations.
\end{corr}

\begin{proof}
By Theorem \ref{thm}, the inequality \eqref{eq:thm-rate-cond-number} holds. Taking logarithm of both sides of this inequality yields
\begin{align}
  \log(F_k) & \leq \log(F_0) + k \log\left( 1-\frac{1}{49Q(K+1)} \right) \nn\\
    & \leq \log(F_0) - \frac{k}{49Q(K+1)}, \nn
\end{align}
where the last line follows since $\log(1+x)\leq x$ for any $x\geq-1$. Therefore, for any $k$ satisfying
\begin{equation}
  \log(F_0)- \frac{k}{49Q(K+1)} \leq \log(\epsilon), \label{eq:opt_soln}
\end{equation}
$x_k$ is an $\epsilon$-optimal solution. Rearranging terms in \eqref{eq:opt_soln}, we conclude that for any $k \geq 49Q(K+1)\log(F_0/\epsilon)$, $x_k$ is an $\epsilon$-optimal solution.
\end{proof}

\bibliographystyle{plain}
\bibliography{piag_ref}

\end{document}